\documentclass[letterpaper, 10 pt, conference]{ieeeconf}  

\IEEEoverridecommandlockouts                              

\pdfoutput=1

\usepackage{grffile}
\usepackage{amsmath,amssymb,amsbsy}
\usepackage{hyperref}

\usepackage{amsthm} 
\usepackage{color}
\usepackage{comment}
\usepackage{lipsum}
\usepackage{graphicx}
\usepackage[normalem]{ulem}
\usepackage{enumerate}
\usepackage{epstopdf}
\usepackage{bbm}
\usepackage{subcaption}

\newtheorem{theorem}{Theorem}[section]
\newtheorem{lemma}[theorem]{Lemma}
	
\newtheorem{proposition}[theorem]{Proposition}
\newtheorem{corollary}[theorem]{Corollary}	
	
\newtheorem{problem}[theorem]{Problem}

\newtheorem{assumption}[theorem]{Assumption}
\newcommand{\ke}[1]{{\color{black} #1}}

\title{\LARGE \bf
Neural ODE Control for Trajectory Approximation\\ of Continuity Equation
}

\author{Karthik Elamvazhuthi, Bahman Gharesifard, Andrea Bertozzi, Stanley Osher
\thanks{*This work was supported by National Science Foundation (NSF) award DMS-1952339 and AFOSR grants FA9550-18-1-0167 and FA9550-18-1-0502.
}
\thanks{$^{1}$Karthik Elamvazhuti is with the Department of Mechanical Engineering, University of California, Riverside, CA 92521, USA {\tt\small kelamvazhuthu@engr.ucr.edu}.
        }%
\thanks{$^{2}$Andrea Bertozzi and Stanley Osher are with the Department of Mathematics, University of California, Los Angeles, CA 90095, USA {\tt\small \{karthikevaz,bertozzi,sjo\}@math.ucla.edu}.
        }%
\thanks{$^{3}$Bahman Gharesifard is with the Department of Electrical Engineering, University of California, Los Angeles, CA 90095, USA
        {\tt\small gharesifard@ucla.edu}}%
}

\begin{document}

\maketitle
\thispagestyle{empty}
\pagestyle{empty}

\begin{abstract}
We consider the controllability problem for the continuity equation, 
corresponding to neural ordinary differential equations (ODEs), which describes how a probability measure is pushedforward by the flow. We show that the controlled continuity equation has very strong controllability properties. Particularly, a given solution of the continuity equation corresponding to a bounded Lipschitz vector field defines
a trajectory on the set of probability measures. For this trajectory, we show that there exist piecewise constant training weights for a neural ODE such that the solution of the continuity equation corresponding to the neural ODE is arbitrarily close to it. As a corollary to this result, we establish that the continuity equation of the neural ODE is approximately controllable on the set of compactly supported probability measures that are absolutely continuous with respect to the Lebesgue measure. 
\end{abstract}

\section{INTRODUCTION}

In recent years, there has been a considerable amount
of work on deep neural networks, due to the flexibility they provide for training purposes. 
Continuum limits of such neural networks has lead to a wide literature on the so-called
Neural ordinary differential equations (ODEs)~\cite{weinan2017proposal,haber2017stable,chen2018neural}, with a major upshot being that tools from dynamical systems and control theory can be used to understand and develop methods to train and synthesize neural networks. This includes extensions to stochastic settings \cite{tzen2019neural,wang2019resnets}, and higher order dynamical variants \cite{sander2021momentum}.

Control-theoretic tools have recently been utilized to address questions related to training neural networks. 
In \cite{tabuada2020universal}, the authors use differential geometric techniques to establish the controllability properties of the underlying neural ODE and leverage that to obtain uniform approximation results. Controllability properties of neural ODEs on the group of diffeomorphisms has also been investigated in \cite{agrachev2021control}. In addition to this, optimal control theory has been leveraged to train neural networks, and this direction has been an original motivation for neural ODEs, see~\cite{chen2018neural}, and also~\cite{jabir2019mean,benning2019deep,li2018maximum}.




An interesting property that is the focal part of this paper is that 
 data classification and universal approximation capabilities of 
neural ODEs can be related to problems arising in optimal transportation theory~\cite{ambrosio2008gradient},
where one aims to take a given probability density to
another. Such transport problems also naturally arise for density estimation problems in machine learning, such as normalizing flows \cite{kobyzev2020normalizing}. This insight has been used to leverage
tools from optimal transportation theory \cite{ambrosio2008gradient} to find numerically efficient
methods to train neural ODEs \cite{finlay2020train}. The resulting transport problems can be analyzed in terms of a controlled continuity equation, which describes how a probability density evolves under the action of flow of a differential equation. This motivates us to study the approximation capabilities of neural ODEs for density estimation, by studying the control properties of the corresponding continuity equation for which the vector-field is given by that of a neural ODE. Most closely related to our work, the authors in~\cite{ruiz2021neural}
establish approximate controllability of the underlying controlled continuity equation on the space of probability measures. Particularly, given a initial and final measure, it is shown that there exist weight parameters of the neural ODE that can be chosen such that the final condition of the continuity equation is arbitrarily close to the target probability measure in the Wasserstein-$1$ distance. 

\emph{Statement of Contributions:} We study the approximation capabilities of
the continuity equation corresponding to a neural ODE. We show one
can construct a sequence of control inputs such that the solutions of the continuity equation corresponding to the neural ODE  uniformly converge to  the solution of the continuity
equation corresponding to any given Lipschitz vector field. This controllability property of the system is challenging to attain due to the low number of control parameters at hand, attributed to the limited width of the neural network. In general, it is not possible to select the control parameters in a way that the approximating vector fields are \emph{strongly} converging to the original one. The key idea behind our result is the observation that one can instead construct admissible vector fields that are weakly converging to the original one, and more importantly, that this is sufficient to achieve uniform convergence of the curves on the set of probability measures. 

\section{Problem Formulation and Motivation}
\label{sec:prob}
Let $\sigma :\mathbb{R} \rightarrow \mathbb{R}$ be a given {\it activation function}. 
\ke{We define the map $\Sigma : \mathbb{R}^d \rightarrow \mathbb{R}^d$ by
\[\Sigma(x) = [\sigma(x_1),...,\sigma(x_d)]^T.\]}
An example of the class of activation functions that we consider is {\it sigmoidal functions} with globally bounded derivatives. An activation function $\sigma$ is  said to be {\it sigmoidal} if its range lies in $[0,1]$, 
\[
\lim_{x \rightarrow -\infty} \sigma(x)=0 \ \mathrm{and} \ \lim_{x \rightarrow \infty} \sigma(x)=1.
\]
One such sigmoidal function is 
\begin{equation}
\label{eq:log}
    \sigma(x) = \frac{1}{1+e^{-x}}.
\end{equation}
Another important example of an activation function is the {\it Rectified Linear Unit (ReLU)}  function defined by
\begin{equation}
\label{eq:relu}
    \sigma (x) = \begin{cases}
    x & ~~x>0, \\
    0 & {\rm otherwise}.
    \end{cases}
\end{equation}

We consider the neural ODE given by 
\begin{eqnarray}
\dot{x}(t) = A(t)\Sigma(W(t)x+\theta(t)),
\label{eq:node}
\end{eqnarray}
where $A: [0,T] \rightarrow \mathbb{R}^{d \times d}$, $W: [0,T] \rightarrow  \mathbb{R}^d$ and  $\theta : [0,T] \rightarrow \mathbb{R}^d$ are the control inputs or weights for the neural network. See \cite{chen2018neural,haber2017stable,tabuada2020universal,ruiz2021neural} for a discussion on the relation between the above ODE and deep residual neural networks.

Suppose that the initial condition $x(0)$ of~\eqref{eq:node} is chosen at random from a distribution with a probability density function $\rho_0$.
The uncertainty in the state $x(t)$ is determined by the time-dependent probability density $\rho(t)$ which evolves according to the continuity equation,
\begin{eqnarray}
\label{eq:neurtra}
\frac{\partial \rho}{\partial t} + \nabla \cdot \Big( \big(A(t)\Sigma(W(t)x+\theta(t))\big)\rho \Big)=0, \\ \nonumber
\rho(0)= \rho_0.
\end{eqnarray}
In density estimation problems such as the ones considered in~\cite{chen2018neural,finlay2020train}, the goal is to construct weight functions (or control inputs) $A(\cdot), W(\cdot), \theta(\cdot)$ so that the endpoint of the solution $\rho(T)$ of \eqref{eq:neurtra} is approximately equal to an unknown probability distribution $\rho^f$, using available samples of $\rho^f$. From a control-theoretic point of view, it is natural to ask for which class of target distributions $\rho^f$, solutions of \eqref{eq:neurtra} can be controlled to $\rho^f$ within final time $T$. This problem has been recently considered in \cite{ruiz2021neural}, for the special case of ReLU activation functions and $d \geq 2$, where it has been shown that $\rho(T)$ can be made arbitrarily close to any given compactly supported $\rho^f$ in the {\it Wasserstein $1$-metric}. The purpose of this short paper is to consider the more general {\it trajectory approximation} problem stated below.
\begin{problem}
Given a curve on the set of probability densities $t  \mapsto \tilde{\rho}(t)$, can we construct control inputs $A(\cdot),W(\cdot), \theta(\cdot)$ such that the solution of \eqref{eq:neurtra} is arbitrary close to $\tilde{\rho}(t)$ in a suitable sense, for all $t \in [0,T]$?
\end{problem}

We answer this problem affirmatively in Theorem \ref{thm:MR} for the case when the curve $ t  \mapsto \tilde{\rho}(t)$ is the pushforward of the flow of a uniformly Lipschitz bounded vector field. Then we show that the controllability result in \cite{ruiz2021neural} can be derived as a Corollary to our main result, for general activation functions that satisfy Assumption~\ref{asmp:neura}. \ke{Few motivations for considering the more general trajectory approximation problem include interpolation of data lying on the set of probability measures \cite{chen2018measure}, identifying dynamical systems from population data \cite{weinreb2018fundamental}, and control of large swarms \cite{elamvazhuthi2019mean}.}

\section{Notation and Preliminaries}
\label{sec:not}

In this section, we define some notation that will be used throughout the paper. We refer the readers to~\cite{ambrosio2008gradient} for more details.
Let $\mathcal{P}_2(\mathbb{R}^d)$ denote the set of Borel probability measures on $\mathbb{R}^d$ with finite second moment: $\int_{\Omega} |x|^2d\mu(x)~<~\infty$. For a given Borel map $T : \mathbb{R}^d \rightarrow  \mathbb{R}^d$ we will denote by \ke{$T_{\#}$ the corresponding pushforward map, which maps any measure $\mu $ to a measure $T_{\#}\mu$, where $T_{\#}\mu$ is the measure defined by
\begin{equation}
    (T_\# \mu)(B) = \mu(T^{-1}(B)), 
\end{equation}
for all Borel measurable sets $B \subseteq \mathbb{R}^d$.} For $\mu,\nu \in \mathcal{P}_2(\mathbb{R}^d)$, we denote the set of transport plans from $\mu$ to $\nu$ by
\begin{equation}
    \Gamma(\mu,\nu):=\{\gamma \in \mathcal{P}(\mathbb{R}^d \times \mathbb{R}^d) | \pi^1_{\#}\gamma = \mu, \pi^2_{\#}\gamma = \nu \},
\end{equation}
where $\pi^i:\mathbb{R}^d \times \mathbb{R}^d \rightarrow \mathbb{R}^d$ are the projections on to the $i$th coordinates, respectively.
We will define the $2-$Wasserstein distance between two probability measures $\mu,\nu$
as the following 
\begin{equation}
    W_2(\mu,\nu) = \min_{\gamma \in \Gamma(\mu,\nu)} \Bigg (\int_{\mathbb{R}^d \times \mathbb{R}^d} |x-y|^2d\gamma(x,y)\Bigg )^{1/2}.
\end{equation}
The set $C([0,T],\mathcal{P}_2(\mathbb{R}^d))$ will refer to the set of continuous curves $t \mapsto \mu_t$ in $\mathcal{P}_2(\mathbb{R}^d)$ with respect to the topology induced by the $2-$Wasserstein distance. We will say that a sequence $\{\mu^N\}_{N \in \mathbb{Z}_+} $ in $C([0,T];\mathcal{P}_2(\mathbb{R}^d))$ converges to $\mu \in C([0,T];\mathcal{P}_2(\mathbb{R}^d))$ if $\lim_{N \rightarrow \infty} \sup_{t\in [0,T]}W_2(\mu^N_t,\mu_t) =0$. Given a vector-field $V : [0,T] \times \mathbb{R}^d \rightarrow \mathbb{R}^d$,  
we will consider solutions $\mu$ of the continuity equation, 
\begin{eqnarray}
\label{eq:cteq}
\frac{\partial \mu}{\partial t} + \nabla \cdot \big(V_t(x) \mu \big)=0,  \\ \nonumber
\mu(0)= \mu_0.
\end{eqnarray}
Naturally, we say that $\mu \in C([0,T],\mathcal{P}_2(\mathbb{R}^d))$ is a weak solution, \ke{or a {\it solution in the sense of distributions} \cite[Section 8.1]{ambrosio2008gradient}}
of the continuity equation \eqref{eq:cteq} if 
\begin{eqnarray}
\int_0^T \int_{\mathbb{R}^d} \bigg (\frac{\partial \phi(t,x)} {\partial t} + \nabla \phi (t,x) \cdot V_t(x) \bigg) d \mu_t(x) dt  \nonumber \\
= -\int_{\mathbb{R}^d} \phi(0,x)d\mu_0(x),
\end{eqnarray}
for all compactly supported real-valued functions $\phi \in C^\infty([0,T) \times \mathbb{R}^d)$.
We make the following assumption.
\begin{assumption}
The vector field $V:[0,T] \times \mathbb{R}^d \rightarrow \mathbb{R}^d$ is such that $t \mapsto V_t(x)$ is measurable for every $x \in \mathbb{R}^d$ and it is uniformly Lipschitz in $x$. That is, there exists $K>0$ such that
\[|V_t(x) -V_t(y)| \leq K|x-y|,\]
for all $x,y \in \mathbb{R}^d$ and all $t \in [0,T]$.
\label{asmp:vec}
\end{assumption}
In addition to this, we will need some mild assumptions on the activation function $\sigma:\mathbb{R}\rightarrow \mathbb{R}$.
For this purpose, let us define the set of functions  
  \begin{align*}
    \mathcal{F}=\bigcup_{m\in \mathbb{Z}_+}\{ \sum_{i=1}^m \alpha_i \sigma(w_i^Tx+\theta_i) \ | \ \alpha_i \in \mathbb{R}, w_i \in \mathbb{R}^d, \theta_i \in \mathbb{R}\}.
    \end{align*}
    \ke{Note that the set $\mathcal{F}$ is the set of arbitrarily wide single-hidden layer neural networks.}
\begin{assumption}
\label{asmp:neura}
We make the following assumptions:
\begin{enumerate}
    \item \textbf{(Regularity)} The activation function $\sigma$ is globally Lipschitz, that is, there exists $K>0$ such that 
    \begin{equation}
    |\sigma (x) - \sigma (y)| \leq K|x-y|, 
     \label{asmp:neura1}
    \end{equation}
     for all $x,y\in \mathbb{R}$.
    \item \textbf{(Density of superpositions)} The set of functions  $ \mathcal{F}$
    is dense in $C(\mathbb{R}^d;\mathbb{R})$ in the uniform norm topology on compact sets. Particularly, given a function $f \in C(\mathbb{R}^d;\mathbb{R})$, for each compact set $\Omega \subset \mathbb{R}$ and $\delta>0$, there exists a function $g \in \mathcal{F}$ such that 
    \[\sup_{x\in \Omega} |f(x) -g(x)| <\delta.\]     \label{asmp:neura2}
\end{enumerate}
\end{assumption} 
It is well-known that the Logistic function \eqref{eq:log} and the ReLU function \eqref{eq:relu} satisfy the density property, see~\cite{cybenko1989approximation,leshno1993multilayer}.
Given Assumption \ref{asmp:neura}, it is easy to see that the subset of vector-valued functions $ \mathcal{F}_d$ defined by
    \begin{align*}
    \mathcal{F}_d=\bigcup_{m\in \mathbb{Z}_+}\{ \sum_{i=1}^m A_i \Sigma(W_ix+\theta_i) \ | \ A_i , W_i \in \mathbb{R}^{d\times d}, \theta_i \in \mathbb{R}^d\},
    \end{align*}
is dense in $C(\mathbb{R}^d;\mathbb{R}^d)$ in the uniform norm topology on compact sets. 

\section{Analysis}
\label{sec:ana}
In this section, we perform our controllability analysis. We show that given  a solution of the continuity equation \eqref{eq:cteq}, we can approximate the solution arbitrarily well using solutions of the equation \eqref{eq:neurtra}. 
\begin{theorem}
\label{thm:MR}
\textbf{(Main Result)}
Suppose that Assumptions~\ref{asmp:neura} and~\ref{asmp:vec} hold and $\mu_0 \in \mathcal{P}_2(\mathbb{R}^d)$ has compact support.  Let $\mu$ be the weak solution of the continuity equation~\eqref{eq:cteq} corresponding to the vector field $V$. Additionally, suppose that $V$ is uniformly bounded in space and time.
Then  for every $\epsilon >0$, there exist  piecewise constant control inputs $A^\epsilon(\cdot), W^{\epsilon}(\cdot)$ and $\theta^\epsilon(\cdot)$, such that the corresponding weak solutions $\mu^\epsilon$ of~\eqref{eq:neurtra} satisfy
\begin{equation}
    \sup_{t \in [0,T]} W_2(\mu_t^\epsilon,\mu_t) \leq \epsilon.
\end{equation}
\end{theorem}

As a consequence, we obtain the following result which
was established as~\cite[Theorem~5]{ruiz2021neural}. 

\begin{corollary}
\label{cor:appctr}
\textbf{(Approximate Controllability)}
Suppose that Assumption \ref{asmp:neura} holds and $\mu_0 ,\mu^f\in \mathcal{P}_2(\mathbb{R}^d)$ have compact supports, and are absolutely continuous with respect to the Lebesgue measure. 
Then for every $\epsilon >0$, there exist piecewise constant control inputs  $A^\epsilon(\cdot), W^{\epsilon}(\cdot)$ and $\theta^\epsilon(\cdot)$, such that the corresponding weak solutions $\mu^\epsilon$ of the equation \eqref{eq:neurtra}, satisfy
\begin{equation}
     W_2(\mu_T^\epsilon,\mu^f)  \leq \epsilon.
\end{equation}
\end{corollary}


In order to prove the main result and its corollary, we will need some preliminary results. The idea behind the proof is that due to Assumption \ref{asmp:neura}, the convex closure of the set of admissible vector fields includes $V$. This is a well-known idea in theory of differential inclusions and relaxed controls \cite{fattorini1999infinite}. These existing results are not directly applicable to \eqref{eq:neurtra}. That being said, we adapt the arguments to prove the above results.

We first observe some regularity properties of the solution of the continuity equation \eqref{eq:cteq}, with respect to the time variable, which will be used later to invoke compactness of certain approximating sequences.
\begin{lemma}
\label{thm:Lip}
Suppose that $\mu_0 \in \mathcal{P}_2(\mathbb{R}^d)$ has compact support and  $V$ is a Borel measurable vector field for which  $\mu \in C([0,T];\mathcal{P}_2(\mathbb{R}^d))$ is a solution of the continuity equation. Suppose there exists $C>0$ such that $|V_t(x) | \leq C $ for $\mu_t$ almost every $ x \in \mathbb{R}^d$, for (Lebesgue) almost every $t \in (0,T)$. Then the curve $\mu$ is Lipschitz:
\begin{equation}
  W_2(\mu_t,\mu_s) \leq K |t-s|,
\end{equation}
for all $0 < t \leq s < T$, where $K$ is a positive constant that depends only on $C$. 
\end{lemma}
\begin{proof}
From \cite[Theorem 8.3.1]{ambrosio2008gradient}, we know the curve $\mu$ is absolutely continuous in the sense of \cite[Definition 1.1.1]{ambrosio2008gradient}, since $\int_{\mathbb{R}^d}\|V_t(x)\|^2 d\mu_t(x)$ is essentially bounded over $(0,T)$. Moreover, from \cite[Theorem 8.3.1]{ambrosio2008gradient}, the {\it metric derivative} of $\mu$ defined by 
\[|\mu'(t)|:= \lim_{s\rightarrow t} \frac{W_2(\mu_t,\mu_s)}{|t-s|}, \]
is essentially bounded by $ \big (\int_{\mathbb{R}^d}\|V_t\|^2 d\mu_t(x)\big)^{1/2}$ . Since, $|V_t| \leq C $ for $\mu_t$ almost every $\mathbb{R}^d$, for (Lebesgue) almost every $t \in (0,T)$, from \cite[Theorems 1.1.2 and 8.3.1]{ambrosio2008gradient}, we have that $W_2(\mu(t),\mu(s))  \leq  \int^t_s |\mu'(\tau)|d\tau \leq K |t-s|$  for all $0 <t \leq s < T$. This concludes the result.
\end{proof}

Next, we observe some classical properties on the relation between solutions of the  continuity equation  \eqref{eq:cteq} and an associated ODE. This result enables some control on the growth of the support of the solution of the continuity equation, \ke{due to the Caratheodory existence theorem for solutions of ODEs.} This, again, will be used later to establish the compactness of certain sequence of curves on $\mathcal{P}_2(\mathbb{R}^d)$. In what follows, we denote by $B_c(0) := \{x \in \mathbb{R}^d; |x| \leq c \}$ the closed ball of radius $c>0$ centered at the origin.

\begin{proposition}
\label{prop:propsu}
Suppose that Assumption \ref{asmp:vec} holds and $\mu_0 \in \mathcal{P}_2(\mathbb{R}^d)$ has a compact support. If $r,R,C> 0$  are such that ${\rm supp}~ \mu_0 \subset B_r(0)$, $|V_t(x)| <C$ for all $(t,x) \in [0,T] \times B_{R+r}(0)$ and $T <\frac{R+r}{C}$. Then there exists a unique solution $\mu $ to the continuity equation \eqref{eq:cteq}. Additionally, the solution $\mu$ is given by $\mu_t= (X_t)_{\#}\mu_0$ for all $t\in [0,T]$, where $X:[0,T] \times \mathbb{R}^d \rightarrow \mathbb{R}^d$ is such that 
\begin{equation*}
    \frac{dX_t(x)}{dt} = V_t(X_t(x));~~~X_0(x)=x.
\end{equation*}
Moreover, ${\rm supp}~ \mu_t \subset B_{R+r}(0)$  for all $t \in [0,T]$.
\end{proposition}
\ke{
\begin{proof}
Due to Assumption \ref{asmp:vec}, for each $x \in B_r(0)$, there exists a unique local solution $y(t)$, of the differential equation 
\begin{equation*}
    \frac{dy(t)}{dt} = V_t(y(t));~~~y(0)=x.
\end{equation*}
From the assumption that there exist $r,R,C> 0$ such that $x \in B_r(0)$, $|V_t(x)| <C$ for all $(t,x) \in [0,T] \times B_{R+r}(0)$ and $T <\frac{R+r}{C}$, and Caratheodory's existence theorem on the existence of solutions to ODEs \cite[Chapter 1, Theorem 1]{filippov2013differential}, we can conclude that the solution $y$ of the above ODE is defined over the interval $[0,T]$ and $y(t) \in B_{R+r}(0)$ for all $t \in [0,T]$. Hence, for $\mu_0$ every $x \in \mathbb{R}^d$, the solution of this ODE is well defined over the interval $[0,T]$ and the result then follows from \cite[Lemma 8.1.6]{ambrosio2008gradient}.\end{proof}
}

In the next proposition we prove the straightforward idea that given a vector-field we can approximate the solution of~\eqref{eq:cteq} using piecewise constant in time vector fields.

\begin{proposition}
\label{thm:Linfap}
Suppose that $V$ satisfies Assumption \ref{asmp:vec}, is uniformly bounded in space and time and $\mu_0 \in \mathcal{P}_2(\mathbb{R}^d)$ has a compact support. Then there exists a sequence $ \{V^N\}_{N\in \mathbb{Z}_+} $ of piecewise constant in time vector fields $V^N:[0,T] \times \mathbb{R}^d \rightarrow \mathbb{R}^d$ such that, the sequence of weak solutions $\{\mu^N\}_{N \in \mathbb{Z}_+}$, corresponding to these vector-fields, converges to the weak solution $\mu$ corresponding to the vector field $V$.
\end{proposition}
\begin{proof}
Define $V^N$ by
\begin{equation}
    V^N_t(x) = \begin{cases}
    \frac{N}{T}\int_{\frac{(n-1)T}{N}}^{\frac{nT}{N}} V_\tau(x)d\tau;  t \in [\frac{(n-1)T}{N},\frac{nT}{N})  \\ ~~~~~~~~~~~~~~~~~~{\rm  for }~n=1,..,N-1, \\
    \frac{N}{T}\int_{\frac{(N-1)}{T}}^T V_\tau(x)d\tau ; t \in [\frac{(N-1)T}{N},T],
    \end{cases}
    \nonumber
\end{equation}
for all $x \in \mathbb{R}^d$. By Lemma \ref{lem:wkcon} $t\mapsto V^N_t(x)$ weakly converges to $t\mapsto V_t(x)$ in $L^1(0,1;\mathbb{R}^d)$ for every $x \in \mathbb{R}^d$. Moreover, it is easy to verify that the vector-fields $V^N$ satisfy Assumption \ref{asmp:vec}
Let $X^N$ be the flow corresponding to the vector fields $V^N$, for each $N \in \mathbb{Z}_+$. It follows \cite[Lemma 2.8]{pogodaev2016optimal} that $\mu_t^N = (X_t^N)_{\#}\mu_0$ are converging to $\mu_t = (X_t)_{\#}\mu_0$ in the weak topology of measures, for each $t\in [0,T]$, as $N$ tends to $\infty$. Invoking Proposition \ref{prop:propsu}, that there exists a compact set $\Omega$ such that the supports of $\mu^N_t, \mu_t$ are contained in $\Omega$ for all $t\in [0,T]$ and for all $N \in \mathbb{Z}_+$. Therefore, since convergence in the weak topology is equivalent to the convergence in the $2$-Wasserstein distance for probability measures with compact support \cite[Theorem 6.9]{villani2009optimal}, this implies that $\{\mu^N_t \}_{N \in \mathbb{Z}_+}$ converges to $\mu_t$ in $\mathcal{P}_2(\mathbb{R}^d)$, for each $t \in [0,T]$. Moreover, due to the uniform bound $ |V^N_t(x)| \leq C$, Lemma \ref{thm:Lip} implies that $\{\mu^N\}_{N \in \mathbb{Z}_+}$ are uniformly Lipschitz in the time variable and hence, invoking the Arzelà-Ascoli theorem, there exists a subsequence of $\{\mu^N\}_{N \in \mathbb{Z}_+}$ that is converging to $\tilde{\mu}$ in $C([0,T],\mathcal{P}_2(\mathbb{R}^d))$. But we know that $\{\mu^N_t \}_{N \in \mathbb{Z}_+}$ converges to $\mu_t$ in $\mathcal{P}_2(\mathbb{R}^d)$, for each $t \in [0,T]$. Therefore, it must be that $\tilde{\mu} = \mu$. This concludes the proof.
\end{proof}

\begin{proposition}
\label{thm:apprne}
Suppose that $\mu_0 \in \mathcal{P}_2(\mathbb{R}^d)$ has a compact support, and that $V:[0,T] \times \mathbb{R}^d \rightarrow \mathbb{R}^d$ is uniformly bounded in space and time and satisfies Assumption \ref{asmp:vec}.  Additionally, assume that the vector field is piecewise constant in time. Given Assumption \ref{asmp:neura}, for $N \in \mathbb{Z}_+$, there exist vector fields $ Q^N$ that are piecewise constant and such that  $Q^N_t \in \mathcal{F}_d$ for all $t \in [0,T]$, and the sequence of weak solutions $\{\mu^N\}_{N \in \mathbb{Z}_+}$, corresponding to the vector-fields $\{Q^N \}_{N\in \mathbb{Z}_+}$, converges to the weak solution $\mu$ corresponding to the vector field $V$.

\end{proposition}
\begin{proof} 

Suppose that the support of $\mu_0$ lies in $B_r(0)$ for some $r >0$. Since the vector-field $V$ is uniformly Lipschitz and bounded, the support of $\mu_t$ lies in $B_{R+r}(0)$ for all sufficiently large $R>0$. Choose $R$ such that $T < \frac{R+r}{C+\delta}$ for some $\delta>0$ and the support of $\mu_t$ lies in $B_{R+r}(0)$ for all $t\in [0,T]$. Define $\Omega := B_{R+r}(0) $. 
 By Assumption \ref{asmp:neura}, we can construct approximating vector-fields $Q^N$ such that 
$Q^N$ are piecewise constant in time, for $N \in \mathbb{Z}_+$, $Q^N_t \in \mathcal{F}_d$ for all $t \in [0,T]$, and $\{Q^N\}_{N \in \mathbb{Z}_+}$ strongly converges to $V$ uniformly in time and space on compact sets:
\[\lim_{N \rightarrow \infty } \sup_{(t,x) \in [0,T] \times \Omega} \|V_t(x)-Q^N_t(x)\|_{\infty} =0.\]
and $|Q^N(t,x)| < C+\delta$ for all $(t,x) \in [0,T] \times \Omega$ and all $N \in \mathbb{Z}_+$.
We can conclude $\mu^N_t$ is contained in $\Omega$ for all $t\in [0,T]$ and all $N \in \mathbb{Z}_+$ , due to Proposition \ref{prop:propsu}. Due to the uniform bound on the velocity fields on $\Omega$, Lemma \ref{thm:Lip} implies that $\{\mu^N\}_{N \in \mathbb{Z}_+}$ are uniformly Lipschitz in time.
Therefore, there exists a subsequence of $\{\mu^N\}_{N \in \mathbb{Z}_+}$ that converges to a limit $\tilde{\mu}$ in $C([0,T];\mathcal{P}_2(\mathbb{R}^d))$. Next, we will verify that $\tilde{\mu}$ is the weak solution of the continuity equation \eqref{eq:cteq} corresponding to the curve $V$. Let $\phi \in C^{\infty}([0,T) \times \mathbb{R}^d)$ be a compactly supported function. Since the supports of $\mu^N$ and $\tilde{\mu}$ are contained in the compact set $\Omega$, 
\begin{eqnarray}
\int_0^T \int_{\mathbb{R}^d} \bigg (\frac{\partial \phi(t,x)} {\partial t} + \nabla \phi (t,x) \cdot Q^N_t(x)\bigg) d \mu^N_t(x) dt  \nonumber \\
-\int_0^T \int_{\mathbb{R}^d} \bigg (\frac{\partial \phi(t,x)} {\partial t} + \nabla \phi (t,x) \cdot V_t(x) \bigg) d \tilde{\mu}_t(x) dt = \nonumber
\end{eqnarray}
\begin{eqnarray}
\int_0^T \int_{\Omega} \bigg (\frac{\partial \phi(t,x)} {\partial t} + \nabla \phi (t,x) \cdot Q^N_t(x) \bigg) d \mu^N_t(x) dt  \nonumber \\
-\int_0^T \int_{\Omega} \bigg (\frac{\partial \phi(t,x)} {\partial t} + \nabla \phi (t,x) \cdot V_t(x) \bigg) d \tilde{\mu}_t(x) dt.
\label{eq:conv1}
\end{eqnarray}

Since $\{Q^N\}_{N \in \mathbb{Z}_+}$ is uniformly converging to $V$ on $[0,T] \times \Omega$, we can conclude that the terms $\bigg (\frac{\partial \phi} {\partial t} + \nabla \phi  \cdot Q^N \bigg)$  are uniformly converging to 
$\bigg (\frac{\partial \phi} {\partial t} + \nabla \phi  \cdot V \bigg)$ on $[0,T] \times \Omega$, as $N$ tends to $\infty$. Moreover, the sequence $ \{\mu^N\}_{N \in \mathbb{Z}_+}$ is converging to $\tilde{\mu}$ in $C([0,T];\mathcal{P}_2(\mathbb{R}^d))$.  By an application of the Dominated Convergence Theorem, \eqref{eq:conv1} converges to $0$ as $N$ tends to $\infty$. This implies that $\tilde{\mu}$ is the weak solution of the continuity equation \eqref{eq:cteq} corresponding to the velocity field $V$ since we conclude that
\begin{eqnarray}
\int_0^T \int_{\mathbb{R}^d} \bigg (\frac{\partial \phi(t,x)} {\partial t} + \nabla \phi (t,x) \cdot V_t(x) \bigg) d \tilde{\mu}_t(x) dt  \nonumber \\
= -\int_{\mathbb{R}^d} \phi(0,x)d\mu_0(x), \nonumber
\end{eqnarray}
for all compactly supported functions $\phi \in C^\infty([0,T) \times \mathbb{R}^d)$. The solution $\mu$ to the continuity equation \eqref{eq:cteq} is unique. Therefore, $\tilde{\mu} = \mu$.
\end{proof}

Next, we show that given a vector field that is a superposition of functions of the form $\Sigma(W\cdot +\theta)$, we can construct an oscillating sequence of admissible vector fields that converge to the superposition weakly in time. 

\begin{lemma}
\label{thm:weaapp}
Let $A_i,W_i \in \mathbb{R}^{d\times d}, \theta_i \in \mathbb{R}^d$ \ke{be weight parameters for $i= 1,...,m$.} For each $N \in \mathbb{Z}_+$. Let $Q^N$ be a $\frac{T}{N}$-periodic vector field defined by 
\begin{equation}
\label{eq:defosc}
\ke{Q_{t+\frac{nT}{N}}(x)} =  m A_i\Sigma(W_ix+\theta_i), ~~ t \in [\frac{iT}{mN},\frac{(i+1)T}{mN}), 
\end{equation}
for all $n \in \{0,...,N-1\}$, $i \in \{ 0,1,...,m-1\}$ and $x \in \mathbb{R}^d$. Then, for each $x \in \mathbb{R}^d$, $t \mapsto Q^N_t(x)$ weakly converges to  $\sum_{i=1}^m A_i\Sigma(W_ix+\theta_i)$  in $L^1(0,T;\mathbb{R}^d)$ for all $x \in \mathbb{R}^d$, as $N$ tends to $\infty$. 

\end{lemma}
\begin{proof}
We note that $\ke{\frac{1}{T}\int_0^T  Q^N_t(x)} = \sum_{i=1}^mA_i\sigma(W_ix+\theta_i) $ for all $x \in \mathbb{R}^d$ and all $N \in \mathbb{Z}_+$. The weak convergence of \ke{$t \mapsto Q^N_t(x)$} to $\sum_{i=1}^m A_i\Sigma(W_ix+\theta_i)$  in $L^1(0,T;\mathbb{R}^d)$, for each $x \in \mathbb{R}^d$, as $N$ tends to $\infty$, follows from \cite[Theorem 8.2]{chipot2009elliptic}. Note that the latter result is stated of functions that are $p$-integrable for $p>1$. However, since $\sum_{i=1}^m A_i\Sigma(W_i x+\theta_i)$ is essentially bounded, the result applies. 
\end{proof}

\begin{proposition}
\label{prop:peri}
Let $\mu_0 \in \mathcal{P}_2(\mathbb{R}^d)$ have a compact support. 
Suppose Assumption \ref{asmp:neura} holds. 
\ke{Let $Q :[0,T]\times \mathbb{R}^d \rightarrow \mathbb{R}^d$} be a piecewise constant in time vector field  such that \ke{$Q_t \in \mathcal{F}_d $} for all $t\in [0,T]$.
Then there exist vector fields \ke{$Q^N:[0,T]\times \mathbb{R}^d \rightarrow \mathbb{R}^d$} that are of the form of the right hand side of~\eqref{eq:node} for piecewise constant controls $A^N(\cdot) ,W^N(\cdot), \theta^N(\cdot)$ such that the sequence of solutions $ \{\mu^N\}_{N \in \mathbb{Z}_+}$ of \eqref{eq:neurtra} for these choices of weights, converges to the solution $\mu$, corresponding to the vector field \ke{$Q$}, in $C([0,T];\mathcal{P}_2(\mathbb{R}^d))$. 
\end{proposition}
\begin{proof}
 From Lemma \ref{thm:weaapp} it follows that, for $Q$ given, there exist weakly approximating admissible vector-fields $Q^N$, of the form in the right hand side of \eqref{eq:node}, by repeating the construction in \eqref{eq:defosc} over the time intervals on which $Q$ is constant and concatenating the approximating vector fields. Moreover, from the construction in Lemma \ref{thm:weaapp}, the map $t \mapsto Q_t^N(x)$ weakly converges to $t \mapsto Q_t(x)$, for each $x$, in $L^1(0,T;\mathbb{R}^d)$, as $N$ tends to $\infty$. From \cite[Lemma 2.8]{pogodaev2016optimal}, it follows that $\{\mu^N_t\}_{N \in \mathbb{Z}_+}$ converges to $\mu_t$ in $\mathcal{P}_2(\mathbb{R}^d)$, for each $[0,T]$. From the construction of the weakly converging vector fields $Q^N$ in Lemma \ref{thm:Lip}, the vector fields $Q^N$  are uniform bounded on compact sets and therefore, it follows that the curves $\mu^N$ are uniformly Lipschitz in time. As a result, there exists a subsequence of $\{\mu^N\}_{N \in \mathbb{Z}_+}$ converging in $C([0,T];\mathcal{P}_2(\mathbb{R}^d))$. But we have already established that $\{\mu^N_t\}_{N \in \mathbb{Z}_+}$ converges to $\mu_t$ in $\mathcal{P}_2(\mathbb{R}^d)$, for each $[0,T]$. Therefore, the convergence of $\{\mu^N_t\}_{N \in \mathbb{Z}_+}$ to $\mu_t$ must be uniform in the time variable, and hence $\{\mu^N\}_{N \in \mathbb{Z}_+}$ converges to $\mu$ in $C([0,T];\mathcal{P}_2(\mathbb{R}^d))$. 
\end{proof} 

Now, we are ready to prove our main result
on approximate controllability of \eqref{eq:neurtra} about trajectories of \eqref{eq:cteq}.
~~\\ 
~~

\begin{proof}[Proof of Theorem \ref{thm:MR}]
 The result follows by applying Proposition \eqref{thm:Linfap} to approximate $V$ using a vector fields that are piecewise constant in time and, then using \ke{Proposition \ref{prop:peri}} to approximate the piecewise constant approximations using vector fields of the form in the right-hand side of \eqref{eq:node}. 
\end{proof}

Finally, owing to an existing result on the approximate controllability of the continuity equation \eqref{eq:cteq} proved in \cite{duprez2019approximate}, we can establish 
approximate controllability of \eqref{eq:neurtra}.

\begin{proof}[Proof of Corollary \ref{cor:appctr}]
According to \cite[Proposition 3.1]{duprez2019approximate}, it is known that, for every $\epsilon>0$, there exists a uniformly bounded vector field $V$ satisfying Assumption \ref{asmp:vec} such that the solution $\mu$ of \eqref{eq:cteq} satisfies $W_2(\mu_T, \mu_f) \leq \epsilon/2$. Then, due to Theorem \ref{thm:MR}, there exist piecewise constant control inputs  $A^\epsilon(\cdot), W^{\epsilon}(\cdot)$ and $\theta^\epsilon(\cdot)$, such that the corresponding weak solutions $\mu^\epsilon$ of the equation \eqref{eq:neurtra}, satisfies
\begin{equation}
     W_2(\mu_T^\epsilon,\mu_T)  \leq \epsilon/2.
\end{equation}
Using the triangle inequality property of the  $W_2$-distance, we can conclude that $W_2(\mu_T^\epsilon,\mu^f)  \leq \epsilon$. This concludes the proof. 
\end{proof}
\section{conclusion}
We demonstrated how neural ODEs can be used to approximate solutions of the continuity equation with  a uniformly Lipschitz bounded vector field. Interesting future directions include extending the result to vector fields that are not Lipschitz, such as those arising from solution of the Benamou-Brenier formulation of optimal transport. Lastly, one could also consider similar approximation results for stochastic and higher order dynamical variants of neural ODEs.

\section*{acknowledgements}
The authors thank Katy Craig and Levon Nurbekyan for many helpful discussions.

\appendix

\newcounter{mycounter}
\renewcommand{\themycounter}{A.\arabic{mycounter}}
\newtheorem{lemmaappendix}[mycounter]{Lemma}

\begin{lemmaappendix}
\label{lem:wkcon}
Let $f \in L^1(0,T;\mathbb{R}^n)$. Suppose that there exists a constant $C >0$ such that $|f(t)| < C$ for almost every $t \in (0,T)$. For each $N \in \mathbb{Z}_+$, consider $f^N \in L^1(0,T;\mathbb{R}^d)$  defined by 
\begin{equation}
    f^N(t) = 
    \frac{N}{T}\int_{\frac{(n-1)T}{N}}^{\frac{(n)T}{N}} f(\tau)d\tau;~t \in [\frac{(n-1)T}{N},\frac{nT}{N}), 
\end{equation}
for $n=1,..,N $. Then the sequence $\{f^N\}_{N\in \mathbb{Z}_+}$ weakly converges to $f$ in $L^1(0,T;\mathbb{R}^d)$ as $N \rightarrow \infty$.
\end{lemmaappendix}
\begin{proof}
By the Lebesgue differentiation theorem  \cite[Theorem 2.3.4]{hytonen2016analysis}, $\{f^N(t)\}_{N\in \mathbb{Z}_+}$ converges to $f(t)$ for almost every $t \in (0,T)$. Since $|f(t)|$ and $|f^N(t)|$ are bounded by $C$ for almost every $t \in (0,T)$, it follows from the dominated convergence theorem that 
\[\lim_{N \rightarrow \infty}\int_{0}^T |f(t)-f^N(t)|dt = 0.\] Therefore, $\{f^N\}_{N\in \mathbb{Z}_+}$ converges to $f$ in the strong topology in $L^1(0,T;\mathbb{R}^d) $ and hence, also in the weak topology.
\end{proof}

\addtolength{\textheight}{-12cm}   


\bibliographystyle{plain}  
\bibliography{cdc}

\end{document}